\documentclass[10pt]{article}

\usepackage{amssymb}
\usepackage{amstext}
\usepackage{enumerate}
\usepackage{verbatim} 

\usepackage{theorem}
\newtheorem{theo}{}
\newtheorem{theor}[theo]{Theorem}
\newtheorem{prop}[theo]{Proposition}

\newtheorem{lemma}[theo]{Lemma}

\newtheorem{remark}[theo]{Remark}

\newenvironment{proof}
	{\par {\bf Proof:}}
 	{\hfill $\square$ \medskip}

\newcommand{\ds}{\displaystyle}

\newcommand{\mc}{\mathcal}

\newcommand{\mf}{\mathfrak}

\newcommand{\la}{\langle}
\newcommand{\ra}{\rangle}

\newcommand{\om}{\omega}
\newcommand{\lam}{\lambda}
\newcommand{\vp}{\varphi}

\newcommand{\be}{\begin{equation}}
\newcommand{\ee}{\end{equation}}

\newcommand\B{\mathcal B}

\renewcommand{\H}{\mc{H}}
\newcommand\LL{\mathcal L}

\newcommand{\M}{\mc{M}}
\newcommand{\RR}{\mc{R}}
\newcommand{\U}{\mc{U}}

\newcommand{\Ds}{\mc{D}^\sharp}

\newcommand{\Pss}{\mc{P}^\sharp}

\newcommand{\Df}{\mc{D}^\flat}

\newcommand{\Pff}{\mc{P}^\flat}

\newcommand\R{\mathbb R}

\title{Admissible vectors and Hilbert algebras}
\author{F. G\'omez-Cubillo$^1$, S. Wickramasekara$^2$\footnote{In memoriam.}}

\begin{document}

\maketitle
\begin{center}
{\small $^1$Dpto de An\'alisis Matem\'atico, IMUVa, Universidad de Valladolid, Facultad de Ciencias, 47011 Valladolid, Spain. fgcubill@am.uva.es.} 

{\small $^2$Department of Physics, Grinnell College, Grinnell, IA 50112, United States. wickrama@grinnell.edu.}
\end{center}

\begin{abstract}
Admissible vectors for unitary representations of locally compact groups are  the basis for group-frame and covariant coherent state expansions.
Main tools in the study of admissible vectors have been Plancherel and central integral decomposition, of applicability only under certain separability and semifiniteness restrictions. In this work we present a study of admissible vectors in terms of convolution Hilbert algebras valid for arbitrary unitary representations of general locally compact groups.


\end{abstract}

\section{Introduction} 

Let $G$ be a locally compact group (lc group, for brevity) with fixed left Haar measure $ds$ and let $\pi$ be a (continuous) unitary representation of $G$ on a Hilbert space $\H_\pi$ with inner product $(\cdot|\cdot)$. An element $\eta\in\H_\pi$ is called an {\it admissible vector} for $\{\pi,\H_\pi\}$ if the operator
$$
L_\eta:\H_\pi\to L^2(G),\quad [L_\eta\psi](s)=(\psi|\pi(s)\eta)\,,
$$
is a bounded map and $L_\eta^*L_\eta=I_{\H_\pi}$, where $I_{\H_\pi}$ denotes the identity operator on $\H_\pi$.
Admissible vectors lead to (weak) resolutions of the identity in terms of the orbit $\pi(G)\eta\subset\H_\pi$:
$$
I_{\H_\pi}=\int_G |\pi(s)\eta)(\pi(s)\eta|\,ds\,.
$$
Relations of this type are known as {\it covariant coherent state expansions} in mathematical physics or {\it group frame expansions} in harmonic analysis; see e.g. \cite{AAG14-2,Chr} and references therein. 

Admissible vectors have been extensively studied restricting attention to  special Lie groups \cite{Per72,Per86,AAG14-2}, irreducible representations \cite{DM76,GMP85,AAG14-2}, second countable groups of type I \cite{AFK03,Fuhr05}, unimodular separable groups \cite{Fuhr03,Bek04} and countable discrete groups \cite{HL00,Bek04,BHP15.1}.
Plancherel and central integral decompositions are main tools in these works, which involve severe separability and semifiniteness conditions (type I von Neumann algebras for Plancherel decompositions and type II for central ones). In this work we present a study of admissible vectors in terms of convolution Hilbert algebras valid for arbitrary unitary representations of general lc groups.

Section \ref{sectav} includes basic terminology, notation and well-known results about admissible vectors. 
Section \ref{sectcha} translates the concept of admissible vector to the context of the full right and left convolution Hilbert algebras associated to the lc group $G$. 
Admissible vectors are characterized in Theorems \ref{ipp2} and \ref{ipp2d}. Their relationship is analyzed in Proposition \ref{tav}. Unitary representations with admissible vectors are determined in Theorems \ref{th358} and \ref{cor360}. 
The results of Section \ref{sectcha} are closely related to the topics of {\it functions of positive type} and {\it square-integrable representations} of both groups and Hilbert algebras 
\cite{God48,Am49,Segal50,Rie69,Per71,Gil72,Phi73,Phi75,Dix77,Rie04,Bek04,Stet07,AAG14-2} and {\it reproducing kernel Hilbert spaces} \cite{Ar50,Carey76,Carey77,Carey78,AAG14-2}. In these works similar results  can be found in diverse contexts, using different techniques and under a great variety of conditions. But, until our knowledge, there is not in the literature a concise approach to admissible vectors for arbitrary unitary representations of general lc groups. The results in Section \ref{sectcha} suggest that perhaps the best way to achieve such concise general approach must consider Hilbert algebras. 
Section \ref{sectsfw} explores the left-right duality of admissible vectors driven by the modular conjugation for a nonunimodular lc group $G$ 
(Theorems \ref{ipp2} and \ref{ipp2d} reflect such duality).
It is shown how the {\it standard form} \cite{Haage75,TII} of the left (and right) von Neumann algebra of $G$ dissolves the duality in such a way  that each admissible vector is associated with a (unique) cyclic and separating vector for the corresponding reduced von Neumann algebras relative to their center; see Theorem \ref{thcon3}. 
Finally, Theorem \ref{t9.1.12x} describes the very natural form of the {\it orthogonality relations} in this context.
Several connections with previous concepts and results in the literature shall be detailed in the Remarks along the work.
Examples of application shall be given elsewhere.

\section{Admissible vectors}\label{sectav} 

Let $G$ be a lc group with fixed {left Haar measure} $ds$ and {modular function} $\delta_G$. 
Recall that $\delta_G$ is a continuous homomorphism from $G$ to the multiplicative group $\R^+$ that satisfies the following relations:
\be\label{emf}
d(st)=\delta_G(t)\,ds\,,\quad d(s^{-1})=\delta_G(s^{-1})\,ds\,,\quad s,t\in G\,.
\ee
Let $C_c(G)$, $C(G)$ and $L^2(G)$ denote the spaces of complex-valued functions on $G$, respectively, continuous with compact support, continuous and square-integrable with respect to $ds$.

By a {\it unitary representation} of $G$ we mean a homomorphism $\pi$ from $G$ into the group $\U(\H_\pi)$ of unitary operators on some nonzero Hilbert space $\H_\pi$ that is continuous with respect to 
the strong (weak) operator topology \cite{F95}. 
We shall use the abbreviated notation $\{\pi,\H_\pi\}$ for the representation.
A vector $f\in\H_\pi$ is called a {\it cyclic vector} for $\{\pi,\H_\pi\}$ if $\la\pi(G)f\ra=\H_\pi$, where $\la\cdot\ra$ means $\overline{span}\{\cdot\}$. 
The {\it left} and {\it right regular representations} $\lam$ and $\rho$ of $G$ on $L^2(G)$ are defined by
\be\label{edlrr}
[\lam(s)f](t):=f(s^{-1}t),\quad s\in G,\,f\in L^2(G)\,,
\ee
\be\label{edrr}
[\rho(s)f](t):=\delta_G^{1/2}(s)f(ts),\quad s\in G,\,f\in L^2(G)\,.
\ee
The {\it left} and {\it right group von Neumann algebras} of $G$ shall be denoted by $\LL_G$ and $\RR_G$, i.e.,
\be\label{dflvna}
\LL_G:=\{\lam(s):s\in G\}'',\quad \RR_G:=\{\rho(s):s\in G\}''
\ee
(double commutant). One has $\LL_G'=\RR_G$. See e.g. \cite[Chapter VII]{TII}. 

The main objects of interest in this work are the {\it admissible vectors} for  unitary representations $\{\pi,\H_\pi\}$ of a lc group $G$. They were defined in the Introduction.
The following facts are well-known \cite{Rie69,Fuhr03}.
For the sake of completeness we include a very simple proof.

\begin{prop}\label{ill}
Let $G$ be a lc group and let $\{\pi,\H_\pi\}$ be a unitary representation of $G$. The following are equivalent:
\begin{enumerate}[(i)]
\item 
$\eta$ is an admissible vector for $\{\pi,\H_\pi\}$.
\item
$L_\eta$ is an isometry from $\H_\pi$ into $L^2(G)$.
\item
For every $\psi,\phi\in\H_\pi$,
\be\label{csexp}
\int_G (\psi|\pi(s)\eta)(\pi(s)\eta|\phi)\,ds=(\psi|\phi)\,.
\ee
\item
The range $L_\eta\H_\pi$ of $L_\eta$ is a closed invariant subspace of $L^2(G)$ for the left regular representation $\lam$ of $G$, $\pi$ is equivalent to the subrepresentation $\lam_{|L_\eta\H_\pi}$ and $L_\eta\eta$ is an admissible vector for $\lam_{|L_\eta\H_\pi}$.
\end{enumerate}
In particular, an admissible vector $\eta$ is a cyclic vector for $\{\pi,\H_\pi\}$.
\end{prop}

\begin{proof}
Since the conditions (1) $L_\eta^*L_\eta=I_{\H_\pi}$, (2) $(L_\eta^*L_\eta\psi|\phi)=(\psi,\phi)$ for all $\psi,\phi\in\H_\pi$, (3) $(L_\eta^*L_\eta\psi|\psi)=(\psi,\psi)$ for all $\psi\in\H_\pi$, (4) $||L_\eta\psi||^2=||\psi||^2$ for all $\psi\in\H_\pi$ are mutually equivalent (the equivalence between (2) and (3) follows by polarization), $\eta$ is admissible if and only if $L_\eta$ is an isometry. In such case, the range $L_\eta\H_\pi$ of $L_\eta$ is a closed subspace of $L_2(G)$ and $L_\eta L_\eta^*=P_{L_\eta\H_\pi}$, the orthogonal projection onto $L_\eta\H_\pi$. Condition (2) above is just (iii). 

Now, by the definition of $L_\eta$, for $\psi\in\H_\pi$ and $s,t\in G$,
\begin{eqnarray*}
[\lam(s)L_\eta \psi](t)&=&[L_\eta \psi](s^{-1}t)=(\psi|\pi(s^{-1}t)\eta)\\
&=&(\pi(s)\psi|\pi(t)\eta)=[L_\eta\pi(s)\psi](t)\,,
\end{eqnarray*}
so that
\be\label{iplp}
\lam(s)L_\eta=L_\eta\pi(s),\quad s\in G\,.
\ee
Taking adjoints,
$$
L_\eta^*\lam(s)=\pi(s)L_\eta^*,\quad s\in G\,.
$$
Thus,
$$
P_{L_\eta\H_\pi}\lam(s)=\lam(s)P_{L_\eta\H_\pi},\quad s\in G\,.
$$
This means that the range $L_\eta\H_\pi$ is a closed invariant subspace of the left regular representation $\lam$ and, moreover, since $L_\eta:\H_\pi\to L_\eta\H_\pi$ is unitary, one has the equivalence between (i) and (iv).

Now, suppose that $\eta$ admissible vector for $\{\pi,\H_\pi\}$ which is not a cyclic vector. Then, there exists $0\neq\psi\in \la\pi(G)\eta\ra^\perp$, so that $(\psi|\pi(s)\eta)=0$ for all $s\in G$ and (\ref{csexp}) cannot be satisfied. 
\end{proof}

\begin{remark}\rm
Equation (\ref{csexp}) leads to an expression for $L_\eta^*$ (the ``partial inverse" of $L_\eta$): since every $f\in L_\eta\H_\pi\subset L^2(G)$ is of the form $f(s)=(\psi|\pi(s)\eta)$ for some $\psi\in\H_\pi$,
$$
L_\eta^*(f)=\left\{\begin{array}{ll} \ds\int_G f(s)\,\pi(s)\eta\,ds,&\text{ if }f\in L_\eta\H_\pi\,,\\0,&\text{ if }f\in (L_\eta\H_\pi)^\perp\,,
\end{array}\right.
$$
where the integral must be interpreted in weak sense.
On the other hand, equation (\ref{csexp}) with $\phi=\psi=\eta$ implies that
$||\eta||_{\H_\pi}=\int_G |(\eta|\pi(s)\eta)|^2\,ds$. 
In particular, $\int_G |(\eta|\pi(s)\eta)|^2\,ds<\infty$. This is the original condition in the definition of an admissible vector given by Grossmann, Morlet and Paul in \cite{GMP85}.
Moreover, the constants $c_\eta$ leading to the Duflo-Moore operator \cite{DM76} in \cite[Th.3.1]{GMP85}, are here
$$
c_\eta=\frac{1}{||\eta||_{\H_\pi}}\int_G |(\eta|\pi(s)\eta)|^2\,ds=1\,.
$$
\end{remark}

In what follows,
for an admissible vector $\eta$ for $\{\pi,\H_\pi\}$ we will write
\be\label{enave}
g_\eta:=L_\eta\eta,\quad \H_\eta:=L_\eta\H_\pi\,.
\ee

\section{Admissible vectors and Hilbert algebras}\label{sectcha} 

As before, let $G$ be a lc group with fixed {left Haar measure} $ds$ and {modular function} $\delta_G$. 
In order to study admissible vectors in terms of convolution Hilbert algebras, let us begin by introducing some terminology and notation  mainly borrowed from Takesaki \cite{TII} and Phillips \cite{Phi73}. The convolution product $f\ast g$ and involutions $f\mapsto f^\sharp$ and $f\mapsto f^\flat$ are defined in $C_c(G)$ by 
\be\label{ecsf}
\begin{array}{rcl}
[f\ast g](s)&:=&\int_G f(t)g(t^{-1}s)\,dt,\quad s\in G\,,\\
f^\sharp(s)&:=&\delta_G(s^{-1})\,\overline{f(s^{-1})},\quad s\in G\,,\\
f^\flat(s)&:=&\overline{f(s^{-1})},\quad s\in G\,,
\end{array}
\ee
where the bar denotes complex conjugation. The involutions $f\mapsto f^\sharp$ and $f\mapsto f^\flat$ can be extended to the domains
\be\label{eddisf}
\begin{array}{l}
\Ds:=\big\{f\in L^2(G):\int_G \delta_G(s)\,|f(s)|^2\,ds<\infty\big\}\,,\\[2ex]
\Df:=\big\{f\in L^2(G):\int_G \delta_G^{-1}(s)\,|f(s)|^2\,ds<\infty\big\}\,.
\end{array}
\ee
The corresponding extensions are closed densely defined operators on $L^2(G)$ and bijective involutions on their own domains. It is usual to denote these extensions by $S$ and $F$, i.e.,
\be\label{edisf}
S:\Ds\to\Ds:f\mapsto f^\sharp,\quad F:\Df\to\Df:g\mapsto g^\flat\,.
\ee

The {\it modular operator} $\Delta$ is defined by
\be\label{dmo}
\left.\begin{array}{r}
\Delta:\mc{D}_\Delta\to L^2(G),\quad [\Delta f](s):=\delta_G(s)f(s)\,,\\[1ex]
\ds\mc{D}_\Delta=\big\{f\in L^2(G):\int_G \delta_G^2(s)\,|f(s)|^2\,ds<\infty\big\}\,,
\end{array}\right\}
\ee
and the {\it modular conjugation} $J$ is given by
\be\label{dmc}
J:L^2(G)\to L^2(G),\quad [Jf](s):=\delta_G^{-1/2}(s)\overline{f(s^{-1})}\,.
\ee
The {left} and {right regular representations} $\lam$ and $\rho$ of $G$ are connected by the modular conjugation $J$:
\be\label{llcon1}
J\rho(s)J=\lam(s),\quad J\lam(s)J=\rho(s),\quad s\in G\,.
\ee

A vector $g\in L^2(G)$ is said to be {\it right bounded} (resp. {\it left bounded}) if 
$$
\sup\{||f\ast g||:f\in C_c(G),\,||f||\leq 1\}<\infty
$$ 
$$
\text{(resp. }\sup\{||g\ast f||:f\in C_c(G),\,||f||\leq 1\}<\infty\text{)}\,,
$$ 
where $||\cdot||$ denotes the $L^2(G)$-norm.
The set of all right (resp. left) bounded vectors is denoted by $\B'$ (resp. $\B$). 
Convolution products on the right $f\mapsto f\ast g$ extend to bounded operators $\pi_r(g)$ on $L^2(G)$ for elements $g$ of $\B'$ and, in a similar way, convolution products on the left $f\mapsto g\ast f$ extend to bounded operators $\pi_l(g)$ on $L^2(G)$ for elements $g$ of $\B$, i.e.,
$$
\pi_r(g)f:=f\ast g,\quad g\in\B',\,f\in L^2(G)\,,
$$
$$
\pi_l(g)f:=g\ast f,\quad g\in\B,\,f\in L^2(G)\,.
$$
The sets 
\be\label{dlfha}
\U':=\B'\cap\Df\quad \text{and}\quad \U'':=\B\cap\Ds
\ee 
are the {\it full right} and {\it left convolution Hilbert algebras} of $G$. 
See \cite[Chapter VI]{TII} for details.

An element $e\in\U'$ (resp. $e\in\U''$) is called {\it right self-adjoint idempotent} (resp. {\it left self-adjoint idempotent}) if $e=e^\flat=e\ast e$ (resp. $e=e^\sharp=e\ast e$). Denote by $\mc{E}'$ (resp. $\mc{E}''$) the set of nonzero right (resp. left) self-adjoint idempotents of $\U'$ (resp. $\U''$).
Since $\pi_r(g)^*=\pi_r(g^\flat)$ for $g\in\U'$, one has that $e\in\mc{E}'$ if and only if $\pi_r(e)$ is an orthogonal projection on $L^2(G)$. In a similar way, $e\in\mc{E}''$ if and only if $\pi_l(e)$ is an orthogonal projection on $L^2(G)$.

The following result is written in terms of the involution $F$ and the convolution product --see Equations (\ref{ecsf})--(\ref{edisf})--.
  
\begin{prop}\label{ppcon} 
Let $\eta$ be an admissible vector for $\{\pi,\H_\pi\}$. Then:
\begin{enumerate}[(i)]
\item
$\H_\eta\subset C(G)\cap L^2(G)$.
\item
Let $f\in L^2(G)$. Then, $f\in \H_\eta$ if and only if
$$
f(s)=f\ast g_\eta^\flat(s),\quad s\in G\,.
$$
If $f\in \H_\eta^\perp$, then $f\ast g_\eta^\flat=0$.
\item
$g_\eta=g_\eta\ast g_\eta^\flat=g_\eta^\flat$.
\end{enumerate}
\end{prop}

\begin{proof}
(i) Each $f\in \H_\eta\subset L^2(G)$ is of the form $f(s)=L_\eta\psi=(\psi|\pi(s)\eta)$, $s\in G$, for some $\psi\in\H_\pi$. Since the representation $\pi$ is strong (weak) continuous, $\H_\eta\subset C(G)$

(ii) Let $f\in \H_\eta$. Using the intertwining relation (\ref{iplp}) one gets, for $s\in G$,
\be\label{efcgb}
\begin{array}{rcl}
f(s)&=&(\psi|\pi(s)\eta)=(\psi|L_\eta^*L_\eta\pi(s)\eta)\\[2ex]
&=&(L_\eta\psi|L_\eta\pi(s)\eta)=(L_\eta\psi|\lam(s)L_\eta\eta)\\[2ex]
&=&\ds (f|\lam(s)g_\eta)=\int_G f(t)\overline{g_\eta(s^{-1}t)}\,dt\\[1ex]
&=&\ds\int_G f(t)g_\eta^\flat(t^{-1}s)\,dt=f\ast g_\eta^\flat(s)\,.
\end{array}
\ee
If $f\in \H_\eta^\perp$, since $\H_\eta=\la\lam(s)g_\eta:s\in G\ra$, one has $
f\ast g_\eta^\flat(s)=(f|\lam(s)g_\eta)=0$ for $s\in G$.
Thus, for $f\in L^2(G)$ with $f\notin \H_\eta$, 
$$
f\ast g_\eta^\flat(s)=(f|\lam(s)g_\eta)=(P_{\H_\eta} f|\lam(s)g_\eta)=P_{\H_\eta} f(s)
$$
and the last expression must be different from $f(s)$ for some $s\in G$, since $P_{\H_\eta} f\neq f$. 

(iii) The first equality in (iii) is just (ii) with $f=g_\eta$; the second equality follows from
\be\label{ecfgb}
\begin{array}{rcl}
f\ast g_\eta^\flat(s)&=&\ds\int_G f(t)g_\eta^\flat(t^{-1}s)\,dt\\[2ex]
&=&\ds\int_G f(t)\overline{g_\eta(s^{-1}t)}\,dt\\[2ex]
&=&\ds\int_G f(st)\overline{g_\eta(t)}\,dt\\[2ex]
&=&\ds\int_G \overline{f^\flat(t^{-1}s^{-1})}\,\overline{g_\eta(t)}\,dt\\[2ex]
&=&\ds\overline{g_\eta\ast f^\flat(s^{-1})}=[g_\eta\ast f^\flat]^\flat(s)
\end{array}
\ee
with $f=g_\eta$. Note that these arguments prove implicitly that $g_\eta\in\Df$. Although we use the symbol $g_\eta^\flat$ in (ii), it is not assumed there, neither in (\ref{ecfgb}), that $g_\eta\in\Df$.

\end{proof}

\begin{remark}\label{rmrkhs}\rm
\begin{enumerate}[(a)]
\item
Proposition \ref{ppcon}.(ii) says that the range $\H_\eta$ is a {\it reproducing kernel Hilbert space} \cite{Ar50} with kernel 
$k_\eta(s,t):=g_\eta(s^{-1}t)$.
Obviously, $\H_\eta$ is invariant under the left regular representation $\lam$ of $G$.
\item
By Proposition \ref{ppcon}.(iii), 
$g_\eta\in\mc{P}(G)\cap L^2(G)\subset A(G)$,
where $\mc{P}(G)$ denotes the set of all continuous functions of positive type on $G$ (see e.g. Godement \cite{God48}), and $A(G)$ is the usual {\it Fourier algebra} of $G$ introduced by Eymard \cite{Ey64}. 
Recall that $A(G)$ is identified with the  predual $[\LL_G]_*$ of $\LL_G$ \cite[Th.3.10]{Ey64}. Se also \cite[Section VII.3]{TII} and \cite[Chapter 3]{ES92}.
\end{enumerate}
\end{remark}

The full convolution right Hilbert algebra $\U'$ defined in (\ref{dlfha}) permit us to identify the admissible vectors. The result is a straightforward consequence of Proposition \ref{ppcon}.

\begin{theor}\label{ipp2}
The following are equivalent:
\begin{enumerate}[(i)]
\item 
$\eta$ is an admissible vector for $\{\pi,\H_\pi\}$.
\item
$g_\eta\in \mc{E}'$ and $\pi_r(g_\eta)=P_{\H_\eta}$, where $P_{\H_\eta}$ denotes the orthogonal projection from $L^2(G)$ onto $\H_\eta$.
\end{enumerate}
\end{theor}

\begin{proof}
(i)$\Rightarrow$(ii): If $\eta$ be an admissible vector for $\{\pi,\H_\pi\}$, according to Proposition \ref{ppcon}.(iii), $g_\eta=g_\eta^\flat$, so that $g_\eta\in\Df$. Moreover, by Proposition \ref{ppcon}.(ii), $g_\eta^\flat\in \B'$ and the orthogonal projection $P_{\H_\eta}$ coincides with 
$$ 
\pi_r(g_\eta^\flat)=\pi_r(g_\eta)^*=\pi_r(g_\eta)=\pi_r(g_\eta\ast g_\eta^\flat)=\pi_r(g_\eta)^2.
$$ 

(ii)$\Rightarrow$(i):
If $g_\eta\in\U'$ and $\pi_r(g_\eta)=P_{\H_\eta}$, then $\pi_r(g_\eta)=\pi_r(g_\eta)^*=\pi_r(g_\eta^\flat)$ and
$\pi_r(g_\eta^\flat)f=f$ for $f\in L_\eta\eta$. Since
$$
[\pi_r(g_\eta^\flat)f](s)=[f\ast g_\eta^\flat](s)=(f|\lam(s)g_\eta),\quad s\in G,\,f\in \H_\eta\,,
$$
we have that $g_\eta$ is an admissible vector for $\{\lam_{|\H_\eta},\H_\eta\}$. The equivalence (i)$\Leftrightarrow$(iv) of Proposition \ref{ill} leads to the result.
\end{proof}

\begin{remark}\label{rm345}\rm
Close concepts to admissible vectors, as characterized in Theorem \ref{ipp2}, have been introduced along the literature:
\begin{enumerate}[(a)]
\item
For separable unimodular lc groups $G$, the orthogonal projections in $L^2(G)$ of the form $\pi_r(g)$, $g\in L^2(G)$, were called {\it finite projections} by Segal \cite{Segal50}, who established a generalization of the Plancherel formula using them \cite[Th.3]{Segal50}. 
An equivalent approach was given by Ambrose \cite{Am49} dealing with the so-called {\it $L^2$-systems} and {\it $\H$-systems}, precursors of the concept of Hilbert algebra.
A generalization to Hilbert algebras of the theory of square-integrable representations of unimodular lc groups can be found in Rieffel \cite{Rie69}, where {\it self-adjoint idempotents} in Hilbert algebras are treated along the lines developed by Ambrose \cite{Am49}.
Phillips \cite{Phi73} extended the theory to arbitrary full left Hilbert algebras.  
Stetkaer \cite{Stet07} expanded the study of square-integrable representations to representations induced from a character of a closed subgroup, mainly for unimodular groups. 
\item
For unimodular lc groups $G$, Carey \cite[Lemma 2.5]{Carey78} proved that for   every reproducing kernel Hilbert space $\H\subset L^2(G)$ invariant under the left regular representation $\lam$ of $G$ there exists $g\in L^2(G)$ such that $\pi_r(g)=P_\H$. See Remark \ref{rmrkhs}.b. 
\end{enumerate}
\end{remark}

Now, assume there exist two admissible vectors $\eta_1$ and $\eta_2$ for $\{\pi,\H_\pi\}$.
To simplify notation, let us put
$$
\H_1:=L_{\eta_1}\H_\pi,\quad \H_2:=L_{\eta_2}\H_\pi\,.
$$
We want to analyze the connections between the equivalent representations $\{\lam_{|\H_1},\H_1\}$ and $\{\lam_{|\H_2},\H_2\}$. Obviously, the unitary intertwining operator between them is $L_{\eta_2}L_{\eta_1}^*$.
Given $\psi\in\H_\pi$ we shall write
$$
f=L_{\eta_1}\psi,\quad \tilde f=L_{\eta_2}\psi\,.
$$
so that $\tilde f=L_{\eta_2}L_{\eta_1}^*f$. In particular, for $\eta_1$ and $\eta_2$ we will use the notation:
$$
g_1=L_{\eta_1}\eta_1,\quad \tilde g_1=L_{\eta_2}\eta_1\,,
$$
$$
g_2=L_{\eta_1}\eta_2,\quad \tilde g_2=L_{\eta_2}\eta_2\,.
$$
One has
\be\label{eqtr}
\tilde g_1(s)=(\eta_1|\pi(s)\eta_2)=\overline{(\eta_2|\pi(s^{-1})\eta_1)}=g_2^\flat(s)\,,
\ee
so that $g_2\in\Df$. By construction of $\{\lam_{|\H_2},\H_2\}$ from $\{\lam_{|\H_1},\H_1\}$ and reasoning as in (\ref{efcgb}),
$$
\tilde f(s)=(f|\lam(s)g_2)=f\ast g_2^\flat(s),\quad s\in G,\,f\in \H_1\,.
$$
In particular, by (\ref{eqtr}),
\be\label{eg2tg2}
\tilde g_2=g_2\ast g_2^\flat=\tilde g_1^\flat\ast\tilde g_1\,.
\ee
Interchanging the roles of  $\{\lam_{|\H_1},\H_1\}$ and $\{\lam_{|\H_2},\H_2\}$,
\be\label{eg1g2}
g_1=g_2^\flat\ast g_2\,.
\ee
To summarize,
$$
\begin{array}{rcl}
L_{\eta_2}L_{\eta_1}^*:\H_1&\to& \H_2\\
f=f\ast g_1^\flat=\tilde f\ast \tilde g_1^\flat&\mapsto& \tilde f=f\ast g_2^\flat=\tilde f\ast\tilde g_2^\flat\\
g_1=g_2^\flat\ast g_2&\mapsto& \tilde g_1=g_2^\flat\\
g_2=\tilde g_1^\flat&\mapsto& \tilde g_2=\tilde g_1^\flat\ast\tilde g_1\,.
\end{array}
$$
Assume, in addition, that $g_2=g_2^\flat$. Then, by (\ref{eg2tg2}) and (\ref{eg1g2}),
$$
g_1=g_2^\flat\ast g_2=g_2\ast g_2^\flat=\tilde g_2\,.
$$
This implies that $\pi_r(g_1)=\pi_r(\tilde g_2)$, that is, $P_{\H_1}=P_{\H_2}$ and $\H_1=\H_2$.

In particular, we have proved the following result.

\begin{prop}\label{tav}
Let $\eta$ be an admissible vector for $\{\pi,\H_\pi\}$ and $g_\eta=L_\eta\eta=(\eta|\pi(\cdot)\eta)$. Let $\xi\in \H_\pi$ and put $\xi_\eta:=L_\eta\xi=(\xi|\pi(\cdot)\eta)$. Then 
\begin{enumerate}[(i)]
\item
$\xi$ is an admissible vector for $\{\pi,\H_\pi\}$ if and only if $\xi_\eta\in\Df$ and $g_\eta=\xi_\eta^\flat\ast \xi_\eta$.
\item
If $\xi$ is an admissible vector for $\{\pi,\H_\pi\}$ and $\xi_\eta^\flat=\xi_\eta$, then
$$
(\eta|\pi(\cdot)\eta)=(\xi|\pi(\cdot)\xi)\quad\text{and}\quad\H_\xi=\H_\eta\,.
$$ 
\end{enumerate}  
\end{prop}

We pass to characterize the unitary representations of $G$ with admissible vectors. 
The following results are almost implicit in the proofs of Takesaki \cite[Lemma 3.3]{Tak70} and Phillips \cite[Th.3.5 and Prop.4.2]{Phi73}.

\begin{prop}\label{llrrep}
Let $G$ be a lc group and $\{\lam_{|\H_0},\H_0\}$ a subrepresentation of the left regular representation $\lam$ of $G$. If $\H_0\cap\Df\neq\{0\}$, then
$\H_0\cap\mc{E}'\neq\emptyset$. If, in addition, $\{\lam_{|\H_0},\H_0\}$ is irreducible, then $\H_0\cap\mc{E}'=\{e\}$.
\end{prop}

\begin{proof}
Let $0\neq g\in\H_0\cap\Df$. By \cite[Prop.2.6]{Per71}, $\pi_r(g)$ is closable  and we can consider its closure $A:=\bar\pi_r(g)$. Let 
$A=VH=KV$
be the polar decomposition of $A$, where $H=(A^*A)^{1/2}$ and $K=(AA^*)^{1/2}$. Then $H$ and $K$ are positive and affiliated with $\RR_G=\LL_G'$ and $V\in\RR_G$; see e.g. \cite[Th.6.1.11]{KRII}.
Let $\phi$ be a Borel bounded function of a real variable. By spectral theory, $A\phi(H)$ is a bounded operator on $L^2(G)$. Put
$$
g_\phi:=A\phi(H)g^\flat\,.
$$
We have, for each $f\in C_c(G)$,
\begin{eqnarray*}
(f|g_\phi)&=&(f|A\phi(H)g^\flat) =(\bar\phi(H)A^*f|g^\flat)\\
&=&(\bar\phi(H)\pi_l(f)g^\flat|g^\flat)=(\pi_l(f)\bar\phi(H)g^\flat|g^\flat)\\
&=&(\bar\phi(H)g^\flat|\pi_l(f^\sharp)g^\flat)=
(\bar\phi(H)g^\flat|\pi_r(g^\flat)f^\sharp)\\
&=&(A\bar\phi(H)g^\flat|f^\sharp)\,.
\end{eqnarray*}
Thus, $g_\phi\in\Df$ and $g_\phi^\flat=A\bar\phi(H)g^\flat$; see \cite[VI.1.5.ii]{TII}. 
Futhermore, for $f\in C_c(G)$,
\begin{eqnarray*}
\pi_l(f)g_\phi&=&\pi_l(f)A\phi(H)g^\flat=A\phi(H)\pi_l(f)g^\flat\\
&=&A\phi(H)A^*f=K^2\phi(K)f\,,
\end{eqnarray*}
the last equality due to $A\phi(H)\supseteq \phi(K)A$.
This shows that $g_\phi\in\B'$ and 
$$
\pi_r(g_\phi)=K^2\phi(K)\,.
$$
In particular, if $\alpha\subset(0,\infty)$ is a closed set  such that $\alpha\cap\sigma(H)\neq\emptyset$ and we consider the function $\phi(t)=\chi_\alpha(t)/t^2$, $t\in\R$, then, $g_\phi\in \U'$, $g_\phi=g_\phi^\flat$ and $\pi_r(g_\phi)=E_K(\alpha)$, where $E_K$ denotes the spectral measure associated to $K$. Thus, $g_\phi\in\mc{E}'$.
Moreover, $g_\phi\in \overline{\text{rang}}(A)=\overline{\text{rang}}(\bar\pi_r(g))=\la\bar\pi_r(g)\U''\ra=\la\pi_l(\U'')g\ra$ (recall that $\la\cdot\ra$ means $\overline{span}\{\cdot\}$). Since $\H_0$ is invariant under $\LL_G$ and $\pi_l(\U'')\subset\LL_G$, we get $g_\phi\in\H_0$.

Now, assume in addition that $\{\lam_{|\H_0},\H_0\}$ is irreducible. If there exist $e_1\neq e_2$ in $\H_0\cap\mc{E}'$, then, by \cite[Lem.2.2]{Per71}, the projections $\pi_r(e_1)$ and $\pi_r(e_2)$ do not coincide, neither the subspaces  $\pi_r(e_1)(L^2(G))$ and $\pi_r(e_2)(L^2(G))$. Since $\pi_r(e_i)(L^2(G))=\la\pi_l(U'')e_i\ra$, $i=1,2$, are invariant subspaces of $\H_0$ for the left regular representation $\lam$, this contradicts the fact that  $\{\lam_{|\H_0},\H_0\}$ is irreducible. Thus, $\H_0\cap\mc{E}'=\{e\}$.
\end{proof}

As in the proof of Proposition \ref{llrrep}, for $0\neq g\in\Df$ we consider the closure $\bar\pi_r(g)$ of $\pi_r(g)$ and
$$
|\bar\pi_r(g)|:=(\bar\pi_r(g)^*\bar\pi_r(g))^{1/2}\,.
$$
Functional calculus derived from spectral theory is applied to $|\bar\pi_r(g)|$
.
\begin{theor}\label{th358}
Let $\{\pi,\H_\pi\}$ be a unitary representation of a lc group $G$. The following assertions are equivalent:
\begin{enumerate}[(i)]
\item
$\{\pi,\H_\pi\}$ has an admissible vector.
\item
$\{\pi,\H_\pi\}$ is equivalent to a subrepresentation of $\lam$, $\{\lam_{|\H_0},\H_0\}$, with a cyclic vector $g\in\Df$ such that $0$ does not belong to the spectrum $\sigma(|\bar\pi_r(g)|)$ of $|\bar\pi_r(g)|$ or $0$ is an isolated point of $\sigma(|\bar\pi_r(g)|)$.
\end{enumerate}  
In such case, $\bar\pi_r(g)|\bar\pi_r(g)|^{-2}g^\flat$ is an admissible vector for $\{\lam_{|\H_0},\H_0\}$.
\end{theor}

\begin{proof}
(i)$\Rightarrow$(ii): 
This implication follows from Theorem \ref{ipp2}. Indeed, put $g=g_\eta$ and $\H_0=\H_\eta$. Since $\pi_r(g)=|\pi_r(g)|=P_{\H_0}$, one has $\sigma(|\pi_r(g)|)=\{1\}$ or $\sigma(|\pi_r(g)|)=\{0,1\}$, and, since $g\in\H_0$, 
$$
\pi_r(g)|\pi_r(g)|^{-2}g^\flat=\pi_r(g)^{-1}g=P_{\H_0}^{-1}g=g\,.
$$ 

(ii)$\Rightarrow$(i):
Put $A:=\bar\pi_r(g)$ and consider its polar decomposition
$A=VH=KV$.  
Let $\phi$ be the function of a real variable satisfying $t^2\phi(t)=\chi_\alpha(t)$, where $\chi_\alpha$ denotes the characteristic function of a closed set $\alpha\subset(0,\infty)$ such that $\sigma(\bar\pi_r(g))\backslash\{0\}\subseteq\alpha$. 
Let us consider 
$$
g_\phi:=A\phi(H)g^\flat\,.
$$
Then, reasoning as in the proof of Proposition \ref{llrrep}, we deduce that  
$g_\phi\in\mc{E}'$ and $\pi_r(g_\phi)=P_{\H_\phi}$, where $\H_\phi=\la\pi_l(\U'')g\ra$. That $\la\pi_l(\U'')g\ra=\H_0$ follows from the facts that $g$ is a cyclic vector for $\{\lam_{|\H_0},\H_0\}$ and $\sigma(\bar\pi_r(g))\backslash\{0\}\subseteq\alpha$. Thus, 
$$
g_\phi=AH^{-2}g^\flat=\bar\pi_r(g)|\bar\pi_r(g)|^{-2}g^\flat
$$
is an admissible vector for $\{\lam_{|\H_0},\H_0\}$.
\end{proof}

\begin{theor}\label{cor360}
Let $\{\pi,\H_\pi\}$ be an irreducible unitary representation of a lc group $G$. The following are equivalent:
\begin{enumerate}[(i)]
\item
$\{\pi,\H_\pi\}$ has an admissible vector.
\item
$\{\pi,\H_\pi\}$ is equivalent to an irreducible subrepresentation of $\lam$, $\{\lam_{|\H_0},\H_0\}$, such that $\H_0\cap\Df\neq\emptyset$.
\end{enumerate}  
In such case, $\H_0\cap\mc{E}'=\{e\}$ and $e$ is an admissible vector for $\{\lam_{|\H_0},\H_0\}$.
\end{theor}

\begin{proof}
By Proposition \ref{llrrep}, one has $\H_0\cap\mc{E}'=\{e\}$. Since $\{\lam_{|\H_0},\H_0\}$ is irreducible and $\pi_r(e)(L^2(G))=\la\pi_l(U'')e\ra$ is an invariant subspace of $\H_0$ for $\lam$, one must have $\H_0=\la\pi_l(U'')e\ra$ and, then, $e$ is a right self-adjoint idempotent of $L^2(G)$, which is cyclic for $\{\lam_{|\H_0},\H_0\}$. Thus, $e$ is an admissible vector for $\{\lam_{|\H_0},\H_0\}$.
\end{proof}

\begin{remark}\rm
\begin{enumerate}[(a)]
\item
If $e_1,e_2\in\mc{E}'$, then one writes $e_1\leq e_2$ if $\pi_r(e_1)\pi_r(e_2)=\pi_r(e_2)\pi_r(e_1)=\pi_r(e_1)$. An element $e\in\mc{E}'$ is said to be {\it minimal} if whenever $e_1\in\mc{E}'$ and $e_1\leq e$, then $e_1=e$.
In Theorem \ref{cor360}, being $\{\lam_{|\H_0},\H_0\}$ irreducible, the admissible vector $e$ must be minimal. See also Phillips \cite{Phi75}.
\item
In his work on integrable and proper actions on C$^*$-algebras, Rieffel \cite[Prop.8.7]{Rie04} observes that, for square-integrable irreducible representations $\{\pi,\H_\pi\}$ of a lc group $G$ and normalized $\pi$-bounded vectors $\eta$, right convolution by the coordinate functional $c_{\eta}(s):=(\eta|\pi(s)\eta)$, $s\in G$, is a projection of $L^2(G)$ onto a closed subspace $\H_0$ consisting entirely of continuous functions on which $\lam$ is unitarily equivalent to $\pi$. Obviously, $c_{\eta}$ is what we call here an admissible vector for $\{\lam_{|\H_0},\H_0\}$.
\end{enumerate}
\end{remark}

We can {\it dualize} the above discussion entirely using the modular conjugation $J$ defined in (\ref{dmc}) and the right regular representation $\rho$ given in (\ref{edrr}). This dualization will be relevant in Section \ref{sectsfw}.
The next two results are the respective dual versions of Proposition \ref{ppcon} and Theorem \ref{ipp2}.

\begin{prop}\label{ppcon1} 
Let $\eta$ be an admissible vector for $\{\pi,\H_\pi\}$. Then:
\begin{enumerate}[(i)]
\item
$\pi_l(Jg_\eta)=J\pi_r(g_\eta)J=JP_{\H_\eta}J$ is an orthogonal projection. Let us put
$$
\tilde\H_\eta:=JP_{\H_\eta}JL^2(G)\,.
$$
\item
$J$ is an antiunitary operator from $\H_\eta$ onto $\tilde\H_\eta$.
\item
$Jg_\eta\in\Ds$ and $[Jg_\eta]^\sharp=SJg_\eta=Jg_\eta$.
\item
For $f\in\H_\eta$ and $s\in G$,
$$
Jf(s)=\delta_G^{1/2}(s^{-1})(Jf|\rho(s^{-1})Jg_\eta)=[Jg_\eta\ast Jf](s)\,.
$$
\item
$Jg_\eta=[Jg_\eta]^\sharp\ast[Jg_\eta]=[Jg_\eta]\ast[Jg_\eta]^\sharp=[Jg_\eta]^\sharp$.
\end{enumerate}
\end{prop}

\begin{proof}
(i) Since $JJ=I$, one has 
$$
[JP_{\H_\eta}J]^2=JP_{\H_\eta}JJP_{\H_\eta}J=JP_{\H_\eta}J\,.
$$
Moreover, for $f,g\in L^2(G)$,
$$
(JP_{\H_\eta}Jf|g)=(Jg|P_{\H_\eta}Jf)=(P_{\H_\eta}Jg|Jf)=(f|JP_{\H_\eta}Jg)\,,
$$
that is, $[JP_{\H_\eta}J]^*=JP_{\H_\eta}J$.

(ii) For $f\in \H_\eta$ one has $P_{\H_\eta}f=f$ and
$$
Jf=JP_{\H_\eta}f=JP_{\H_\eta}JJf=P_{\tilde\H_\eta}Jf\,.
$$
Now interchange the roles of $\H_\eta$ and $\tilde\H_\eta$.

(iii) Since $S=\Delta^{-1/2}J$, $F=\Delta^{1/2}J$ and $Fg_\eta=g_\eta$, 
$$
SJg_\eta=\Delta^{-1/2}g_\eta=\Delta^{-1/2}Fg_\eta=Jg_\eta\,.
$$

(iv) By (\ref{llcon1}) and Proposition \ref{ppcon}.(ii), for $f\in\H_\eta$ and $s\in G$,
\begin{eqnarray*}
\delta_G^{1/2}(s^{-1})(Jf|\rho(s^{-1})Jg_\eta)&=&\delta_G^{-1/2}(s)(J\rho(s^{-1})Jg_\eta|f)\\
&=&\delta_G^{-1/2}(s)(\lam(s^{-1})g_\eta|f)\\
&=&\delta_G^{-1/2}(s)\overline{f(s^{-1})}=Jf(s)\,.
\end{eqnarray*}
On the other hand, using (\ref{emf}) and Proposition \ref{ppcon}.(iii), for $f\in\H_\eta$ and $s\in G$,
\begin{eqnarray*}
\delta_G^{1/2}(s^{-1})(Jf|\rho(s^{-1})Jg_\eta)=
\ds\delta_G^{1/2}(s^{-1})\int_G Jf(t)\overline{\rho(s^{-1})Jg_\eta(t)}\,dt\\
=\ds\delta_G^{1/2}(s^{-1})\int_G Jf(t)\overline{\delta^{1/2}(s^{-1})Jg_\eta(ts^{-1})}\,dt\\
=\ds\int_G Jf(t)\overline{Jg_\eta(ts^{-1})}\,d(ts^{-1})=
\int_G Jf(ts)\overline{Jg_\eta(t)}\,d(t)\\
\ds=\int_G Jf(ts)\overline{SJg_\eta(t)}\,d(t)=
\int_G Jf(ts)\delta_G(t^{-1})Jg_\eta(t^{-1})\,d(t)\\
\ds=\int_G Jf(ts)Jg_\eta(t^{-1})\,d(t^{-1})=
\int_G Jf(t^{-1}s)Jg_\eta(t)\,d(t)\\
=[Jg_\eta\ast Jf](s)\,.
\end{eqnarray*}

(v) The result follows from items (iii) and (iv).
\end{proof}

\begin{theor}\label{ipp2d}
The following are equivalent:
\begin{enumerate}[(i)]
\item 
$\eta$ is an admissible vector for $\{\pi,\H_\pi\}$.
\item
$Jg_\eta\in \mc{E}''$ and $\pi_l(Jg_\eta)=P_{J\H_\eta}$, where $P_{J\H_\eta}$ denotes the orthogonal projection from $L^2(G)$ onto $J\H_\eta$.
\end{enumerate}
\end{theor}

\begin{proof}
The result is proved using items (ii) and (iv) of Proposition \ref{ppcon1} like in Theorem \ref{ipp2}. It is also a direct consequence of Theorem \ref{ipp2} and Tomita-Takesaki formulas
$$\begin{array}{c}
\pi_r(Jg)=J\pi_l(g)J,\quad g\in\U''\,,\\
\pi_l(Jg)=J\pi_r(g)J,\quad g\in\U'\,;
\end{array}
$$
see \cite[VI.1.19]{TII}. 
\end{proof}

\begin{remark}\label{rmrkhsd}\rm
Proposition \ref{ppcon1}.(iv) implies that $J\H_\eta$ is a {\it reproducing kernel Hilbert space} \cite{Ar50} with kernel 
$\tilde k_\eta(s,t):=\delta_G^{-1}(s)Jg_\eta(ts^{-1})$.
See Remarks \ref{rmrkhs}.b and \ref{rm345}.b.
\end{remark}

\section{Admissible vectors and standard forms}\label{sectsfw} 

Let $G$ be a lc group and let $\Delta$ and $J$ be the modular operator and  modular conjugation defined, respectively, by (\ref{dmo}) and (\ref{dmc}).
Let $\{\pi,\H_\pi\}$ be a unitary representation of $G$ with an admissible vector $\eta$ and, as in Section \ref{sectcha}, Eq.(\ref{enave}), let us put $g_\eta:=L_\eta\eta$ and $\H_\eta:=L_\eta\H_\pi$.

Since $\H_\eta$ is invariant under $\lam$ and $J\H_\eta$ is invariant under $\rho$, one has for the corresponding projections that
$P_{\H_\eta}\in \RR_G$ and $P_{J\H_\eta}\in \LL_G$,
where $\LL_G$ and $\RR_G$ are the left and right von Neumann algebras of $G$ given in (\ref{dflvna}).
Let us consider the {\it reduced von Neumann algebras}\footnote{\label{fnrvna} Given a von Neumann algebra $\M$ acting on a Hilbert space $\H$, if $P\in\M$ is a projection, the {\it reduced von Neumann algebra} $\M_P$ is the set of all $A\in\M$ such that $PA=AP=A$.  $\M_P$ is a von Neumann algebra on $P\H$ and its commutant is the induced von Neumann algebra $[\M']_P$  whose elements are all the restrictions to $P\H$ of elements of $\M'$.} generated, respectively, by $\lam$ on $\H_\eta$ and by $\rho$ on $J\H_\eta$: 
$$
\LL_\eta:=\big\{\lam(s)_{|\H_\eta}:s\in G\big\}'',\quad 
\RR_\eta:=\big\{\rho(s)_{|J\H_\eta}:s\in G\big\}''\,. 
$$
Obviously,
$J\LL_\eta J=\RR_\eta$, $J\RR_\eta J=\LL_\eta$ and 
$$
\H_\eta=\la\LL_\eta g_\eta\ra=\la\LL_Gg_\eta\ra,\quad 
J\H_\eta=\la\RR_\eta Jg_\eta\ra=\la\RR_G Jg_\eta\ra\,. 
$$

Now, let us consider the mutually dual\footnote{\label{fnsdc} Recall that for a convex cone $\mf{P}$ in $L^2(G)$, the {\it dual cone} $\mf{P}^\circ$ is defined by $\mf{P}^\circ:=\{g\in L^2(G):(f|g)\geq0 \text{ for }f\in\mf{P}\}$. If $\mf{P}=\mf{P}^\circ$, then $\mf{P}$ is called {\it self-dual}.} pointed convex cones in $L^2(G)$
$$
\Pff:=\big\{g\ast g^\flat:g\in \U'\big\}^-,\quad
\Pss:=\big\{g\ast g^\sharp:g\in \U''\big\}^-\,, 
$$
where the bar means the closure. Then,
\be\label{dfsdcc}
\mf{P}:=(\Delta^{-1/4}\Pff)^-=(\Delta^{1/4}\Pss)^-
\ee
is a self-dual closed convex cone of $L^2(G)$. 
Moreover, every element of $L^2(G)$ is represented as a linear combination of four vectors of $\mf{P}$ and to each positive form $\om$ in the predual $[\LL_G]_*$ there corresponds a unique $g\in \mf{P}$ with $\om=\om_g$, i.e.,
$$
\om(A)=\om_g(A):=(Ag|g),\quad A\in\LL_G\,.
$$
The  quadruple $\big\{\LL_G,L^2(G),J,\mf{P}\big\}$ is a {\it standard form}\footnote{Standard forms were introduced by Haagerup \cite{Haage75}. Every von Neumann algebra $\M$ can be represented on a Hilbert space $\H$ in which it is standard. In such standard form every $\om$ in the predual $\M_*$ is of the form $\om =\om_{\xi,\eta}$, i.e., $\om(A)=\om_{\xi,\eta}(A):=(A\xi|\eta)_\H$, $A\in\M$, for certain $\xi,\eta\in \H$;
furthermore, every automorphism $\alpha$ of $\M$ is implemented by a unique 
unitary operator $U$ defined on $\H$, that is, $\alpha(A)=UAU^*$, $A\in\M$, such that $UJ=JU$ and $U\mf{P}\subset\mf{P}$. See also \cite[Sect.IX.1]{TII}. 
} of the von Neumann algebra $\LL_G$, that is, the following requirements are satisfied: 
\begin{enumerate}[(i)]
\item
$J\LL_GJ=\LL_G'=\RR_G$, 
\item
$JAJ=A^*$, for $A\in \LL_G\cap\RR_G$, 
\item
$Jg=g$, for $g\in\mf{P}$, 
\item
$AJAJ\mf{P}\subset\mf{P}$, for $A\in\LL_G$. 
\end{enumerate}

Our next objective is Theorem \ref{thcon3}, the main result of the work. It says that for each admissible vector $\eta$ the object of interest is not $g_\eta$ nor $Jg_\eta$, but $\Delta^{-1/4}g_\eta=\Delta^{1/4}Jg_\eta=J\Delta^{-1/4}g_\eta$, which satisfies notable properties. Among other things, this function is cyclic and separating for the corresponding reduced left and right von Neumann algebras, determines standard forms for such algebras and defines a faithful finite normal weight on them. 

\begin{lemma}\label{lcon3} 
Let $\eta$ be an admissible vector for $\{\pi,\H_\pi\}$. Then:
\begin{enumerate}[(i)]
\item
$g_\eta\in \Pff$ and $Jg_\eta\in \Pss$.
\item
$g_\eta\in\mc{D}(\Delta^{-1/4})$, $Jg_\eta\in\mc{D}(\Delta^{1/4})$ and
\be\label{eeamgpjg}
\Delta^{-1/4}g_\eta=\Delta^{1/4}Jg_\eta=J\Delta^{-1/4}g_\eta\,.
\ee
\item
$\la \LL_G\Delta^{-1/4}g_\eta\ra=\la \RR_G\Delta^{1/4}Jg_\eta\ra=\la \RR_G\Delta^{-1/4}g_\eta\ra$.
\end{enumerate}
\end{lemma}

\begin{proof}
(i) follows from Proposition \ref{ppcon}.(iii) and  Proposition \ref{ppcon1}.(v).

(ii) Since $g_\eta\in\Df=\mc{D}(\Delta^{-1/2})$ and $\mc{D}(\Delta^{-1/2})$ is a core for $\Delta^{-1/4}$ (see e.g. \cite[2.7.7]{KRI}), then $g_\eta\in\mc{D}(\Delta^{-1/4})$. In a similar way, $Jg_\eta\in\Ds=\mc{D}(\Delta^{1/2})$ implies $Jg_\eta\in\mc{D}(\Delta^{1/4})$. 
Now, since $g_\eta=g_\eta^\flat=Fg_\eta$ (Proposition \ref{ppcon}.iii) and $F=\Delta^{1/2}J$ \cite[VI.1.5]{TII},
$$
\Delta^{-1/4}g_\eta=\Delta^{-1/4}Fg_\eta=\Delta^{-1/4}\Delta^{1/2}Jg_\eta=\Delta^{1/4}Jg_\eta\,,
$$
$$
[J\Delta^{-1/4}g_\eta](t)=\delta_G^{-1/2}(t)\delta_G^{-1/4}(t^{-1})\overline{g_\eta(t^{-1})}=[\Delta^{-1/4}g_\eta](t)\,.
$$

(iii) Since $J$ is an (antilinear) isometry and $J^2=I$, one has, by (\ref{llcon1}) and (\ref{eeamgpjg}), 
\begin{eqnarray*}
J\la \LL_G\Delta^{-1/4}g_\eta\ra
&=&J\la \{\lam(s)\Delta^{-1/4}g_\eta:s\in G\}\ra\\
&=&\la \{J\lam(s)JJ\Delta^{-1/4}g_\eta:s\in G\}\ra\\
&=&\la \{\rho(s)J\Delta^{-1/4}g_\eta:s\in G\}\ra\\
&=&\la \{\rho(s)\Delta^{-1/4}g_\eta:s\in G\}\ra\\
&=&\la \{\rho(s)\Delta^{1/4}Jg_\eta:s\in G\}\ra\,,
\end{eqnarray*}
that is,
\be\label{eqlrs}
J\la \LL_G\Delta^{-1/4}g_\eta\ra=\la \RR_G\Delta^{1/4}Jg_\eta\ra=\la \RR_G\Delta^{-1/4}g_\eta\ra\,.
\ee
On the other hand, by (ii), $g_\eta\in\mc{D}(\Delta^{-1/4})$ and, thus, for $s\in G$, the function $t\in G\mapsto \delta_G^{-1/4}(s^{-1}t)g_\eta(s^{-1}t)$ is in $L^2(G)$. But
\begin{eqnarray*}
\delta_G^{-1/4}(s^{-1}t)g_\eta(s^{-1}t)&=&\delta_G^{1/4}(s)\delta_G^{-1/4}(t)g_\eta(s^{-1}t)\\
&=&\delta_G^{1/4}(s)[\Delta^{-1/4}\lam(s)g_\eta](t)
\end{eqnarray*}
and this implies that $\lam(s)g_\eta\in \mc{D}(\Delta^{-1/4})$ and 
$$
\lam(s)\Delta^{-1/4}g_\eta=\delta_G^{1/4}(s)\Delta^{-1/4}\lam(s)g_\eta,\quad s\in G\,.
$$
Thus, 
\begin{eqnarray*}
\la \LL_G\Delta^{-1/4}g_\eta\ra&=&\la \{\lam(s)\Delta^{-1/4}g_\eta:s\in G\}\ra\\
&=&\la \{\Delta^{-1/4}\lam(s)g_\eta:s\in G\}\ra
\end{eqnarray*}
and, then, using again that $g_\eta=g_\eta^\flat$,
\begin{eqnarray*}
J\la \LL_G\Delta^{-1/4}g_\eta\ra
&=&\la \{J\Delta^{-1/4}\lam(s)g_\eta:s\in G\}\ra\\
&=&\la \{\Delta^{-1/2}\Delta^{1/4}\lam(s)g_\eta:s\in G\}\ra\\
&=&\la \{\Delta^{-1/4}\lam(s)g_\eta:s\in G\}\ra\\
&=&\la \LL_G\Delta^{-1/4}g_\eta\ra\,.
\end{eqnarray*}
This equality, together with (\ref{eqlrs}), lead to the result.
\end{proof}

Thus, if $\eta$ is an admissible vector for $\{\pi,\H_\pi\}$, according to Lemma \ref{lcon3}.iii, we can consider the closed subspace $\hat\H_\eta$ of $L^2(G)$ defined by
\be\label{dfhhe}
\hat\H_\eta:=\la \LL_G\Delta^{-1/4}g_\eta\ra=\la \RR_G\Delta^{-1/4}g_\eta\ra\,.
\ee
Obviously, the corresponding orthogonal projection $P_{\hat\H_\eta}$ from $L^2(G)$ onto $\hat\H_\eta$ belongs to the center $\LL_G\cap\RR_G$ and the elements of the reduced von Neumann algebras $\hat\LL_\eta$ and $\hat\RR_\eta$ of $\LL_G$ and $\RR_G$ to $\hat\H_\eta$ (see footnote \ref{fnrvna}) are just the restrictions to $\hat\H_\eta$ of elements of $\LL_G$ and $\RR_G$, i.e.,
\be\label{dfrvnahhe}
\begin{array}{l}
\hat\LL_\eta:=\big\{\lam(s)_{|\hat\H_\eta}:s\in G\big\}''=[\LL_G]_{P_{\hat\H_\eta}}=\LL_G{P_{\hat\H_\eta}}\,,\\ 
\hat\RR_\eta:=\big\{\rho(s)_{|\hat\H_\eta}:s\in G\big\}''=[\RR_G]_{P_{\hat\H_\eta}}=\RR_G{P_{\hat\H_\eta}}\,. 
\end{array}
\ee
Both (reduced) von Neumann algebras $\hat\LL_\eta$ and $\hat\RR_\eta$ act on $\hat\H_\eta$ and, by definition, $[\hat\LL_\eta]'=\hat\RR_\eta$.
Moreover, the definition of $\hat\H_\eta$ also implies that $\Delta^{-1/4}g_\eta$ is a cyclic vector for both $\hat\LL_\eta$ and $\hat\RR_\eta$.
Recall that, given a von Neumann algebra $\M$ acting on a Hilbert space $\H$,  an element $\xi\in\H$ is called a {\it separating vector} for $\M$ if, for any $A\in\M$, $A\xi=0$ implies $A=0$. It is well-known that $\xi$ is a separating vector for $\M$ if and only if $\xi$ is a cyclic vector for $\M'$; see e.g. \cite[Prop.2.5.3]{BR87}. Thus, $\Delta^{-1/4}g_\eta$ is a cyclic and separating vector 
of $\hat\H_\eta$ for both $\hat\LL_\eta$ and $\hat\RR_\eta$.

These results and additional ones are included in the next theorem.

\begin{theor}\label{thcon3} 
Let $\eta$ be an admissible vector for $\{\pi,\H_\pi\}$. Let $\hat\H_\eta$ be the closed subspace of $L^2(G)$ defined by (\ref{dfhhe}) and let $\hat\LL_\eta$ and $\hat\RR_\eta$ be the reduced von Neumann algebras  of $\LL_G$ and $\RR_G$ to $\hat\H_\eta$ given in (\ref{dfrvnahhe}). Then:
\begin{enumerate}[(i)]
\item
$P_{\hat\H_\eta}\in\LL_G\cap\RR_G$.
\item
$[\hat\LL_\eta]'=\hat\RR_\eta$.
\item
$\Delta^{-1/4}g_\eta$ is a cyclic and separating vector 
of $\hat\H_\eta$ for both $\hat\LL_\eta$ and $\hat\RR_\eta$.
\item
$[\pi_r(\Delta^{-1/4}g_\eta)]_{|\hat\H_\eta}=I_{|\hat\H_\eta}$. If, in addition, $[\Delta^{-1/4}g_\eta]^\flat=\Delta^{1/4}g_\eta$ belongs to $\hat\H_\eta$, then $\pi_r(\Delta^{-1/4}g_\eta)=P_{\hat\H_\eta}$.
\item
Let $J$ be the modular conjugation in $L^2(G)$ defined by (\ref{dmc}). Then $J$ and $P_{\hat\H_\eta}$ commute. Let us consider
$$
\hat J_\eta:=J_{|\hat\H_\eta}\,.
$$ 
Then $\hat J_\eta$ is an antilinear isometry of $\hat\H_\eta$ onto itself such that $\hat J_\eta^2=I_{|\hat\H_\eta}$.
\item
Let $\mf{P}$ be the self-dual closed convex cone of $L^2(G)$ given by (\ref{dfsdcc}). Let us define
$$
\hat\mf{P}_\eta:=\mf{P}\cap\hat\H_\eta\,.
$$
Then $\hat\mf{P}_\eta$ is a self-dual closed convex cone of $\hat\H_\eta$ and
$$
\hat\mf{P}_\eta=\big\{[\R^+g_\eta-\mf{P}]\cap\mf{P}\big\}^-\,.
$$
\item
$\big\{\hat\LL_\eta,\hat\H_\eta,\hat J_\eta,\hat\mf{P}_\eta\big\}$ is a  standard form of the von Neumann algebra $\hat\LL_\eta$.
\end{enumerate}
\end{theor}

\begin{proof}
(i), (ii) and (iii) have been proved in the preceding comments.

(iv) Since $\pi_r(g_\eta)=P_{\H_\eta}$, one has, for $s,u\in G$,
\begin{eqnarray*}
\ds[\lam(u)\Delta^{-1/4}g_\eta]\ast[\Delta^{-1/4}g_\eta](s)\\
\ds =\int_G [\lam(u)\Delta^{-1/4}g_\eta](t)[\Delta^{-1/4}g_\eta](t^{-1}s)\,dt\\
\ds =\int_G \delta_G^{-1/4}(u^{-1}t)g_\eta(u^{-1}t)\delta_G^{-1/4}(t^{-1}s)g_\eta(t^{-1}s)\,dt\\
\ds =\delta_G^{-1/4}(u^{-1}s)[\lam(u)g_\eta]\ast g_\eta(s)\\
\ds =\delta_G^{-1/4}(u^{-1}s)[\lam(u)g_\eta](s)=[\lam(u)\Delta^{-1/4}g_\eta](s)\,.
\end{eqnarray*}
Thus, $[\pi_r(\Delta^{-1/4}g_\eta)]_{|\hat\H_\eta}=I_{|\hat\H_\eta}$.

Since $g_\eta=g_\eta^\flat$, a simple calculation leads to $[\Delta^{-1/4}g_\eta]^\flat=\Delta^{1/4}g_\eta$. If, in addition, $[\Delta^{-1/4}g_\eta]^\flat\in\hat\H_\eta$, then $\pi_r(\Delta^{-1/4}g_\eta)=P_{\hat\H_\eta}$. Indeed, 
in such case, for $f\in\hat\H_\eta^\perp$ and $s\in G$,
\begin{eqnarray*}
\ds f\ast[\Delta^{-1/4}g_\eta](s)=
\int_G f(t)[\Delta^{-1/4}g_\eta](t^{-1}s)\,dt\\
\ds=\int_G f(t)\delta_G^{-1/4}(t^{-1}s)g_\eta(t^{-1}s)\,dt\\
\ds=\int_G f(t)\delta_G^{1/4}(s^{-1}t)\overline{g_\eta^\flat(s^{-1}t)}\,dt\\
\ds=\int_G f(st)\delta_G^{1/4}(t)\overline{g_\eta^\flat(t)}\,dt=
\int_G f(st)\overline{[\Delta^{-1/4}g_\eta]^\flat(t)}\,dt\\
\ds=(\lam(s^{-1})f|[\Delta^{-1/4}g_\eta]^\flat)=(f|\lam(s)[\Delta^{-1/4}g_\eta]^\flat)=0\,.
\end{eqnarray*}

(v) Since $J^2=I$, by (\ref{llcon1}) and (\ref{eeamgpjg}), for $s\in G$,
$$
J\lam(s)\Delta^{-1/4}g_\eta=J\lam(s)JJ\Delta^{-1/4}g_\eta=\rho(s)J\Delta^{-1/4}g_\eta=\rho(s)\Delta^{-1/4}g_\eta\,,
$$
so that $J$ and $P_{\hat\H_\eta}$ commute and $\hat J_\eta=J_{|\hat\H_\eta}$ is an antilinear isometry of $\hat\H_\eta$ onto itself such that $\hat J_\eta^2=I_{|\hat\H_\eta}$.

Finally, (vi) and (vii) follow from  \cite[IX.1.8]{TII} and \cite[IX.1.11]{TII}.
\end{proof}
  
\begin{remark}\rm
A von Neumann algebra $\M$ of operators acting on a Hilbert space is said to be {\it $\sigma$-finite} if all the collections of 
mutually orthogonal projections in $\M$ have at most a countable cardinality.
It is well-known \cite[Prop.2.5.6]{BR87} that $\M$ admits a cyclic and separating vector if and only if $\M$ is $\sigma$-finite.
Thus, Theorem \ref{thcon3}.iii implies that, if $\eta$ is an admissible vector for $\{\pi,\H_\pi\}$, then the reduced von Neumann algebras $\hat\LL_\eta$ and $\hat\RR_\eta$ are both $\sigma$-finite.
\end{remark}

Finally, we include a sort of {\it orthogonality relations} in this context.

\begin{theor}\label{t9.1.12x}
Let $\eta_i$ be an admissible vector for $\{\pi_i,\H_{\pi_i}\}$, $i=1,2$.
The following conditions are equivalent: 
\begin{enumerate}[(i)]
\item
$\Delta^{-1/4}g_{\eta_1}\perp \Delta^{-1/4}g_{\eta_2}$. 
\item
$\hat\mf{P}_{\eta_1}\perp \hat\mf{P}_{\eta_2}$. 
\item
$\hat\H_{\eta_1}\perp \hat\H_{\eta_2}$. 
\end{enumerate}
\end{theor}

\begin{proof}
There is a one to one correspondence between full right and left Hilbert 
algebras and faithful semi-finite normal weights \cite[Sect.VII.2]{TII}.
Here we pay attention to the full convolution left  Hilbert algebra $\U''$ of a lc group $G$ given in (\ref{dlfha}).
The {\it canonical} or{\it Plancherel weight} $\vp$ corresponding to $\U''$ is defined on the positive cone $\LL_G^+$ of $\LL_G$ by 
$$
\vp(A):=\left\{\begin{array}{ll} {||g||,} &\text{if }A^{1/2}=\pi_l(g),\,g\in\U''\,,\\
+\infty,& \text{otherwise}\,;\end{array}\right.
$$
see \cite[VII.2.5]{TII}.
The associated modular operator $\Delta$ and modular conjugation $J$ are those ones defined, respectively, in  (\ref{dmo}) and (\ref{dmc}). The set 
$\mf{P}_\vp:=\{A\in\LL_G^+:\vp(A)<\infty\}$ coincides with the cone $\mf{P}$ defined in (\ref{dfsdcc}) via the identification $\pi_l(g)\leftrightarrow g$, $g\in\U''$.
By the definition of $\mf{P}$ and Lemma \ref{lcon3}.(i), for an admissible vector $\eta$ one has $\Delta^{-1/4}g_{\eta}\in\mf{P}=\mf{P}_\vp$. Then the result is a direct consequence of \cite[IX.1.12]{TII}.
\end{proof}

\begin{remark}\rm
Phillips \cite{Phi75} studies orthogonality relations for irreducible square-integrable representations of left Hilbert algebras similar to those given by Duflo and Moore \cite{DM76} for irreducible unitary representations of nonunimodular groups. See, in particular, \cite[Th.2.4]{Phi75} and \cite[Cor.2.5]{Phi75}. See also Grossmann, Morlet and Paul \cite{GMP85}. See also the reproducing kernel Hilbert space approach given by Carey \cite{Carey76}. Additional comments on orthogonality relations can be found in Rieffel \cite[Sect.8]{Rie04}.
\end{remark}


\section*{Acknowledgments}
This work was supported by research project MTM2014-57129-C2-1-P (Secretar\'{\i}a General de Ciencia, Tecnolog\'{\i}a e Innovaci\'on, Ministerio de Econom\'{\i}a y Competitividad, Spain).


\end{document}